\documentclass{article}
\usepackage[utf8]{inputenc}
\usepackage{amsfonts}
\usepackage{hyperref}
\usepackage{enumitem}
\usepackage{amsmath}
\usepackage{amsthm}
\usepackage[capitalize,nameinlink]{cleveref}
\usepackage{comment}
\usepackage{mathrsfs}
\usepackage{mathabx}
\usepackage{tikz}
\usepackage{todonotes}
\usetikzlibrary{positioning}
\usepackage{caption}
\newtheorem{theorem}{Theorem}[section]
\usepackage[a4paper, margin=1in]{geometry}
\newtheorem{corollary}[theorem]{Corollary}
\newtheorem{lemma}[theorem]{Lemma}

\newtheorem{proposition}[theorem]{Proposition}

\newtheorem{definition}[theorem]{Definition}
\newtheorem{example}[theorem]{Example}

\title{Translation Invariant Operators on Polyanalytic Sobolev-Fock Spaces}
\author{
  Henry McNulty 
  }
\begin{document}

\maketitle
\begin{abstract}
    We examine translation invariant operators on the Polyanalytic Sobolev-Fock spaces and show that they take the form
    \begin{align*}
        S_{\phi} F(z) =  \int_{\mathbb{C}^n} F(w)e^{\pi z\cdot \overline{w}}\phi(w-z,\overline{w}-z) e^{\pi |w|^2}\, dw
    \end{align*}
    for certain $\phi$, using tools from time-frequency analysis. This extends the results of \cite{CaLiShWiYa20} to both the Sobolev-Fock spaces and the Polyanalytic Sobolev-Fock spaces. We use results on symbol classes of pseudo-differential operators to give sufficient conditions for boundedness of $S_{\phi}$ on all polyanalytic Sobolev-Fock spaces.
\end{abstract}

\section{Introduction}
The Fock space $\mathcal{F}^2(\mathbb{C}^d)$ is the space of analytic functions $F$ satisfying the condition
\begin{align*}
    \|F\|_{\mathcal{F}^2} := \Big(\int_{\mathbb{C}^d} |F(z)|^2 e^{-\pi |z|^2}\, dz\Big)^{1/2} < \infty,
\end{align*}
and is a Reproducing Kernel Hilbert Space (RKHS) equipped with inner product
\begin{align*}
    \langle F, G\rangle_{\mathcal{F}^2} := \int_{\mathbb{C}^d} F(z)\overline{G(z)}e^{-\pi |z|^2}\, dz,
\end{align*}
and can be mapped to from $L^2(\mathbb{R}^{2d})$ by the unitary Bargmann transform, given by
\begin{align*}
    Bf(z) := 2^{d/4}e^{\pi z^2/2}\int_{\mathbb{R}^d}f(x)e^{\pi(x-z)^2}\, dx
\end{align*}
for $f\in L^2(\mathbb{R}^d)$.

The Fock space is of interest in several fields such as complex analysis, physics and signal processing, due to its well understood sampling and interpolation properties. In their study of the Bargmann transform \cite{Zh15}, the author, motivated by the form of the Hilbert transform upon conjugating with the Bargmann transform, proposed the problem of classifying the functions $\phi$ for which the operator 
\begin{align}\label{Sformintro}
    S F (z) := \int_{\mathbb{C}^d} F(w)e^{\pi z\overline{w}}\phi(z-\overline{w})e^{-\pi |w|^2}\, dz
\end{align}
is bounded on the Fock space. By recognising the Hilbert transform as a Fourier multiplier, this was solved by \cite{CaLiShWiYa20}. In this work we examine such operators on generalisations of the Fock space. The tools to do so are provided by the field of time-frequency analysis. From this perspective, the Bargmann transform is associated to the Short-Time Fourier Transform (STFT) with respect to a Gaussian window, and we see that by interpreting the operators of the form \cref{Sformintro} as convolutions, we can use the natural framework of time-frequency shifts. The connection to time-frequency shifts, or the Weyl representation of the Heisenberg group, was made in relation to operators of the form \cref{Sformintro} in \cite{Th23-1} \cite{Th23-2}, where the twisted Fock space and Sobolev-Fock spaces were considered. Further generalisations of the result of \cite{CaLiShWiYa20} appear in \cite{BaNa23}, \cite{GuHeLiSh22} and \cite{WiWu22}. We generalise the results on Fock space in two ways: Firstly, we take the Sobolev-Fock spaces by extending the Bargmann transform to take arguments in modulation spaces $M^p(\mathbb{R}^d)$, which result in analytic functions satisfying
\begin{align}\label{sobfockspace}
    \int_{\mathbb{C}^d} |F(z)|^p e^{-\pi p|z|^2}\, dz  < \infty.
\end{align}
The modulation spaces are central objects in time-frequency analysis, and enjoy a reproducing structure upon taking the STFT in a similar manner to the Hilbert space $L^2$, which is critical to our results. 

Secondly, we consider the Polyanalytic Sobolev-Fock spaces, corresponding to the STFT with any Hermite function window as opposed to only the Gaussian \cite{Va00}. The resulting functions are then polyanalytic as opposed to analytic in the regular Fock case, satisfying the generalised Cauchy–Riemann equation
\begin{align*}
    \frac{d^k}{d\overline{z}^k} F(z)=0.
\end{align*}

In both cases we use results on translation invariant operators on modulation spaces formulated in \cite{FeNa06}, \cite{BeGrOkRo07}, and classify all operators which are translation invariant in the sense formulated in \cref{adjintertw}. We see that for all $1\leq p < \infty$, any operator taking the form \cref{Sform} is translation invariant and corresponds to a function which acts as a bounded convolution operator on the modulation space, and conversely that any such function gives rise to a corresponding operator $S$ on the Sobolev-Fock space. We do so by recognising the Bargmann transform as an isomorphism between modulation spaces and Sobolev-Fock spaces, and accordingly considering the effect of translations on modulation spaces on the corresponding Sobolev-Fock space.

By recognising that these tools are acting on the modulation spaces, we then extend the results to the polyanalytic Sobolev-Fock spaces, and find the correspondance between functions defining a bounded convolution operator on the modulation space with bounded operators of the form
\begin{align}\label{Sformpolyintro}
    S F (z) = \int_{\mathbb{C}^d} F(w)e^{\pi \overline{w}z} \psi(z-w,z-\overline{w}) e^{-\pi |w|^2}\, dw.
\end{align}
In a sense this indicates the privileged status of the Fock space, since this is the only space for which $\phi$ takes only $z-\overline{w}$ as an argument in such operators. 

Finally we present some results which may be of independent interest, using well-known results from time-frequency analysis connecting the symbols of pseudo-differential operators to their boundedness on modulation spaces. This allows us to give sufficient conditions for boundedness of an operator of type \cref{Sformintro} and \cref{Sformpolyintro} on all polyanalytic Sobolev-Fock spaces in terms of the function $\phi$ and $\psi$ respectively. This approach also gives a sufficient condition to the open problem proposed in \cite{CaLiShWiYa20} concerning general integral kernel operators on the Fock space.  

Throughout this article we use the terminology that a Sobolev-Fock function is one which satisfies \cref{sobfockspace}, in accordance with \cite{KeLu21}. This should not be confused with the Fock-Sobolev spaces discussed in for example \cite{WiWu22}.

\section{Preliminaries}
\subsection{Time-Frequency Analysis tools}
Time-frequency analysis aims to represent functions in joint time-frequency or phase space. We define the translation and modulation operators:
\begin{definition}
    For a function $f\in L^2(\mathbb{R}^{d})$, the translation operator $T_x$ and modulation operator $M_{\omega}$ are defined as 
    \begin{align*}
        T_x f(t) &= f(t-x) \\
        M_{\omega} f(t) &= e^{2\pi i \omega\cdot t}f(t).
    \end{align*}
\end{definition}
The composition of the two then gives a time-frequency shift, $\pi(z)$, for $z=(x,\omega)\in\mathbb{R}^{2d}$. A central object in time-frequency analysis is the Short-Time Fourier Transform (STFT):
\begin{definition}
    Given a function function $f\in L^2(\mathbb{R}^d)$ and a window $g\in L^2(\mathbb{R}^d)$, the STFT $V_g f$ at a point $z$ in phase space is given by 
\begin{align*}
    V_g f(z) = \langle f, \pi(z) g \rangle_{L^2} = \int_{\mathbb{R}^d} f(t)\overline{g(t-x)}e^{-2\pi t\cdot \omega}\, dt.
\end{align*}
\end{definition}
For all such $f,g$, $V_g f \in L^2(\mathbb{R}^{2d})$. The STFT is an isometry for normalised $g$, and the image is in fact a Reproducing Kernel Hilbert Space (RKHS). Since we will later consider Polyanalytic Sobolev-Fock spaces, we consider the extension of the STFT to tempered distributions by duality. For $f\in \mathscr{S}^\prime(\mathbb{R}^d)$ and $g\in \mathscr{S}(\mathbb{R}^d)$;
\begin{align*}
    V_{g} f(z) = \langle f, \pi(z)g \rangle_{\mathscr{S}^\prime,\mathscr{S}},
\end{align*}
where the duality bracket is anti-linear in the second argument. The modulation spaces can then be defined as the functions with STFTs in certain Lebesgue spaces. In this work we restrict our attention to unweighted, non-mixed Lebesgue spaces. We can thus define the modulation spaces \cite{Fe83} as follows:
\begin{definition}
    Let the Gaussian be denoted by $\varphi_0(t)=2^{d/4}e^{-2\pi t^2}$. The modulation space $M^p(\mathbb{R}^d)$ for $1\leq p \leq \infty$ is then given by
    \begin{align*}
        M^p(\mathbb{R}^d) := \{f\in\mathscr{S}(\mathbb{R}^d) : V_{\varphi_0}f \in L^p(\mathbb{R}^{2d})\}.
    \end{align*}
\end{definition}
The modulation spaces are Banach spaces with respect to the norm $\|f\|_{M^p(\mathbb{R}^d)} := \|V_{\varphi_0} f\|_{L^p(\mathbb{R}^d)}$, and a fundamental result in time-frequency analysis is that all windows $g\in M^1$, also known as Feichtinger's algebra, give rise to the same $M^p(\mathbb{R}^d)$ spaces as the Gaussian. the modulation spaces satisfy the inclusion relation $M^p(\mathbb{R}^d) \subset M^q(\mathbb{R}^d)$ for $p<q$ where the inclusion is dense for $q<\infty$, and the relationship between the modulation spaces and Lebesgue spaces is given by
\begin{align}\label{modinclrelation}
\begin{split}
    M^p &\subset L^p \; \mathrm{for}\; p<2  \\
    M^2 &= L^2 \\
    L^p &\subset M^p \; \mathrm{for}\; p>2.  \\
\end{split}
\end{align}
The dual space of $M^p(\mathbb{R}^d)$ for $1\leq p < \infty$ is the Hölder conjugate $p'$, with the duality relation
\begin{align*}
    \langle f, g\rangle_{M^p,M^{p'}} = \langle V_{\varphi_0} f, V_{\varphi_0} f\rangle_{L^p,L^{p'}}.
\end{align*}
A crucial result of modulation spaces is that the reproducing kernel of $V_g(L^2)$ extends to a reproducing kernel for all $M^p(\mathbb{R}^d)$ when one takes a window $g\in M^1(\mathbb{R}^d)$, this is known as the correspondence theorem of coorbit theory \cite{FeGr89}. The Wiener-amalgam spaces are useful in the study of time-frequency representations in giving a local and global description of the space. We define them using a Bounded Uniform Partition of Unity (BUPU) $\Psi = \{\psi_i\}_{i\in I}$ (cf. Definition 1 \cite{FeNa06}).
\begin{definition}
    Let $X(I)$ be a solid Banach space of sequences over index set $I$ for which pointwise convergence implies convergence, containing all finite sequences. Then for some Banach space $B\subset\mathscr{S^\prime}$, 
    \begin{align*}
        W(B,X) := \{f \in \mathscr{S}^{\prime}: \big\|  \|f\cdot\psi_i\|_B \big\|_X < \infty\}.
    \end{align*}
\end{definition}
Wiener-amalgam spaces are Banach spaces with the natural norm. Two results we will use from \cite{FeNa06} are the following:
\begin{theorem}[Theorem 10, \cite{FeNa06}]\label{convchar}
For $1\leq p \leq \infty$, any translation invariant operator $A:M^p(\mathbb{R}^d)\to M^q(\mathbb{R}^d)$, is of
convolution type, in the sense that there exists a unique $u\in M^{\infty}(\mathbb{R}^d)$ such that 
\begin{align*}
    Af = u*f
\end{align*}
for every $f\in M^1(\mathbb{R}^d)$.
\end{theorem}
Since $M^1(\mathbb{R}^d)$ is dense in $M^p(\mathbb{R}^d)$ for $p<\infty$, there is a unique bounded operator corresponding to the above convolution defined for the whole $M^p(\mathbb{R}^d)$ space, which is precisely $A$. We denote by $M_c(M^p)$ the space of tempered distributions corresponding to bounded, translation-invariant operators on $M^p(\mathbb{R}^d)$.
\begin{theorem}[Theorem 16, \cite{FeNa06}]\label{convspace}
    For $1\leq p \leq \infty$,
    \begin{align*}
        M_c(M^p) = \mathscr{F}^{-1} W(\mathscr{F}M_c(L^p(\mathbb{R}^d),l^{\infty}).
    \end{align*}
\end{theorem}

\subsection{The Bargmann Transform and Fock Space}
The Bargmann transform, defined by
\begin{align*}
    Bf (z) &= 2^{\frac{d}{4}}\int_{\mathbb{R}^d} f(x)e^{2\pi x z -  \pi x^2 - \pi z^2/2} \, dx \\
    &= 2^{\frac{d}{4}}e^{\pi z^2/2}\int_{\mathbb{R}^d} f(x)e^{\pi(x-z)^2} \, dx
\end{align*}
maps $L^2(\mathbb{R}^d)$ unitarily to the Fock, or Bargmann-Fock space, the Hilbert space of analytic functions of $\mathbb{C}^d$ satisfying
\begin{align*}
    \|F\|_{\mathcal{F}^2}^2 := \int_{\mathbb{C}^d} |F(z)|^2 e^{-\pi|z|^2} \, dz < \infty,
\end{align*}
with inner product 
\begin{align*}
    \langle F, G \rangle_{\mathcal{F}^2} := \int_{\mathbb{C}^d} F(z)\overline{G(z)} e^{-\pi|z|^2} \, dz.
\end{align*}
The adjoint (and inverse) of the Bargmann transform can be calculated directly, giving
\begin{align*}
    B^* F(x) = 2^{\frac{d}{4}}\int_{\mathbb{R}^d} F(z)e^{2\pi x \overline{z} -  \pi x^2 - \pi \overline{z}^2/2}e^{-\pi|z|^2} \, dz.
\end{align*}
The Fock space is a RKHS, with reproducing kernel $k_z(w)=e^{\pi\overline{z}w}$, such that for all $F\in\mathcal{F}^2(\mathbb{C^d})$;
\begin{align*}
    F(z) = \langle F(w), e^{\pi\overline{z}w}\rangle_{\mathcal{F}^2}.
\end{align*}
The true Polyanalytic Fock spaces
\begin{align*}
    \mathcal{F}^2_k(\mathbb{C}^d) := B^k(L^2(\mathbb{R}^d))
\end{align*}
are of interest since they corresponds to the functions which are polyanalytic of order $k$, that is for which
\begin{align*}
    F(z) = \overline{z}^k G(z)
\end{align*}
for some analytic $G$, and
\begin{align*}
    \int_{\mathbb{C}^d} |F(z)|^2 e^{-\pi|z|^2} \, dz < \infty.
\end{align*}
For $k\in\mathbb{N}$, the Polyanalytic Bargmann transform $B^k$ can be defined as
\begin{align*}
    B^k f(z) := c_k e^{\pi |z|^2} \frac{d^k}{dz^k} e^{-\pi |z|^2} Bf(z)
\end{align*}
where $c_k = \frac{1}{\sqrt{\pi^k k!}}$, and is a unitary mapping from $L^2(\mathbb{R}^d)$ to $\mathcal{F}^2_k(\mathbb{C}^d)$. 

\subsection{Finding Time-Frequency Analysis in the Bargmann Transform}
We now focus on the Bargmann transform's identification with the Short-Time Fourier Transform (STFT) with respect to a Gaussian window. A straight-forward calculation (cf. Proposition 3.4.1 \cite{grochenigtfa}) gives the following:
\begin{proposition}
    Let $f\in L^2(\mathbb{R}^d)$, and denote the Gaussian $\varphi_0(t)=2^{d/4}e^{-2\pi t^2}$. Writing $z=x+i\omega$, we have
    \begin{align*}
        V_{\varphi_0} f (x,-\omega) = e^{i\pi x \omega - \pi |z|^2/2} Bf (z).
    \end{align*}
\end{proposition}
As a unitary transformation, we can consider the image of the Hermite functions by $B$ to find an orthonormal basis in Bargmann-Fock space. Defining the $n$-th Hermite function by 
\begin{align*}
    \varphi_n(t) = c_n e^{\pi t^2}\frac{d^n}{dt^n} e^{-2\pi t^2},
\end{align*}
where $c_n$ is the appropriate constant to normalise the functions in $L^2$, one finds that
\begin{align*}
    B \varphi_n(z) = \sqrt{\frac{\pi^n}{n!}} z^n,
\end{align*}
giving a monomial orthonormal basis for $\mathcal{F}^2(\mathbb{C}^d)$. As an analogue to the time-frequency shifts $\pi(z)$, we define shifts on $\mathcal{F}^2(\mathbb{C}^d)$
\begin{align*}
    \beta_z F(w) = e^{i\pi x \omega - \pi|z|^2/2} e^{\pi z w} F(w-\overline{z}),
\end{align*}
where a simple calculation shows the \textit{intertwining property}
\begin{align*}
    \beta_z Bf = B\pi(z)f,
\end{align*}
where again $z=x+i\omega$.

By following the same identification as above for the STFT and Bargmann transform, the Polyanalytic Bargmann transforms can be described by the STFT with Hermite functions as windows:
\begin{align*}
    B^k f(z) := e^{-i\pi x \omega + \pi |z|^2/2} V_{\varphi_k} f (x,-\omega).
\end{align*}
As noted above, the Bargmann-Fock space is a Reproducing Kernel Hilbert Space (RKHS), with reproducing kernel $k(w,z)=e^{\pi\overline{z}w}$. Differentiating the reproducing formula $k$ times gives
\begin{align}\label{diffformula}
    F^{(k)}(z) = \pi^k \langle F(w), w^ke^{\pi\overline{z}w}\rangle_{\mathcal{F}^2}.
\end{align}
The identification of the (polyanalytic) Bargmann transform with the STFT leads naturally to the idea of (polyanalytic) Sobolev-Fock spaces. These are the spaces satisfying an $L^p$ condition as opposed to $L^2$;
\begin{align*}
    \mathcal{F}^p_k(\mathbb{C}^d) := \Big\{F:F/\overline{z}^k\; \mathrm{analytic,}\; \int_{\mathbb{C}^d} \big|F(z)e^{-\pi|z|^2/2}\big|^p\, dz < \infty \Big\},
\end{align*}
and as such $\mathcal{F}^p_k(\mathbb{C}^d) = B^k(M^p(\mathbb{R}^d))$. The polyanalytic Fock space $\mathcal{F}^2_k(\mathbb{C}^d)$ has the reproducing kernel
\begin{align*}
    k(w,z) = L_{k-1}(\pi |z-w|^2)e^{\pi\overline{z}w},
\end{align*}
where $L_k$ is the $k$-th Laguerre polynomial. Since $\mathcal{F}^2_k(\mathbb{C}^d)$ is the image of the STFT, this reproducing kernel extends to all polyanalytic Sobolev-Fock spaces.

Such spaces have been considered in for example \cite{CaHeHo18}, however their analysis becomes much more simple by considering them as images of modulation spaces. Since $\mathcal{F}^p_k(\mathbb{C}^d) = B^k(M^p(\mathbb{R}^d))$, the main theorem from \cite{CaHeHo18}, which states that
\begin{align*}
    B^k(L^p(\mathbb{R}^d)) \subset \mathcal{F}^p_k(\mathbb{C}^d) \; \mathrm{for}\; p>2   \\
    \mathcal{F}^p_k(\mathbb{C}^d) \subset B^k(L^p(\mathbb{R}^d))\; \mathrm{for}\; p<2   \\,
\end{align*}
follow as a simple corollary of the inclusion relations between modulation spaces \cref{modinclrelation}. A time-frequency approach to analysis of the Bargmann transform is hence a useful endeavor, which will be pursued in the next section.

\section{Translation Invariant Operators on Polyanalytic Sobolev-Fock Spaces}
We begin by considering the action of translation invariant operators on Fock space, as in \cite{CaLiShWiYa20}. This is an equivalent task to considering translation invariant operators on $M^p(\mathbb{R}^d)$ spaces, as made concrete by the following lemma:
\begin{lemma}\label{adjintertw}
    Given a bounded operator $S$ on $\mathcal{F}^p(\mathbb{C}^d)$, $S$ satisfies $\beta_y S = S \beta_y$ for every $y\in\mathbb{R}^d$ if and only if there exists an operator $A$ which is bounded on $M^p(\mathbb{R}^d)$, given by
    \begin{align*}
        A = B^*SB,
    \end{align*}
    such that $AT_y = T_yA$.
\end{lemma}
\begin{proof}
 We need to show that 
 \begin{align*}
     B^*\beta_y F= T_y B^*F = \pi(y,0) B^*F.
 \end{align*}
 This follows the fact that $\beta_z Bf = B\pi(z)f$ for $z=(y,0)$, along with the unitarity of $B$: Let $F = Bf$, then
 \begin{align*}
     B^*\beta_y F &= B^*\beta_y B f \\
     &= B^*B\pi(y,0)f \\
     &= \pi(y,0)B^*F.
 \end{align*}
Hence given an $S$ such that $\beta_y S = S \beta_y$ for every $y\in\mathbb{R}^d$, $A = B^*SB$ has the property $AT_y = T_yA$. Conversely given any bounded operator $S$ on $\mathcal{F}^p(\mathbb{C}^d)$, if $A = B^*SB$ satisfies $AT_y = T_yA$, then 
\begin{align*}
    S \beta_y &= B^*AB\beta_y \\
    &= \beta_y S,
\end{align*}
and the result follows. We note that the above calculation also shows that $B\beta_z^* F= \pi(z)^* B^*F$.

\end{proof}
We will use the term \textit{translation invariant} to describe operators on $\mathcal{F}^p_k(\mathbb{C}^d)$ satisfying the above property $S\beta_y = \beta_y S$ for all $y\in \mathbb{R}^d$. With the operator $\beta_z$ defined for $z\in\mathbb{C}^d$ it is natural to wonder why we don't use the term satisfying $S\beta_z = \beta_z S$ for all $z\in \mathbb{C}^d$, however since the time-frequency shifts $\pi(z)$ are an irreducible representation of the Weyl-Heisenberg group, such operators must be multiples of the identity. 

We can then use \cref{convchar} to construct all translation invariant operators on Sobolev-Fock spaces. The following theorem generalises several of the results of \cite{CaLiShWiYa20} to Sobolev-Fock spaces:
\begin{theorem}\label{focksob}
    Let 
    \begin{align*}
        \phi(z) = \int_{\mathbb{R}^d}u(y)e^{-\pi y^2/2-\pi yz}\, dy,
    \end{align*}
    for some $u\in M_c(M^p(\mathbb{R}^d)$ and $1\leq p < \infty$. Then $S:\mathcal{F}^1(\mathbb{C}^d)\to \mathcal{F}^{\infty}(\mathbb{C}^d)$ defined by
    \begin{align}\label{Sform}
        S F (z) = \int_{\mathbb{C}^d} F(w)e^{\pi \overline{w}z} \phi(\overline{w}-z)e^{-\pi |w|^2}\, dw,
    \end{align}
    for $F\in \mathcal{F}^1(\mathbb{C}^d)$, 
    extends uniquely to a bounded, translation invariant operator on $\mathcal{F}^p(\mathbb{C}^d)$. Conversely, if $S$ is a bounded and translation invariant operator on $\mathcal{F}^p(\mathbb{C}^d)$, then $S$ is of the form \cref{Sform}, with corresponding $u\in M_c(M^p(\mathbb{R}^d))$.
\end{theorem}
\begin{proof}
Let $S:\mathcal{F}^1(\mathbb{C}^d)\to \mathcal{F}^{\infty}(\mathbb{C}^d)$ extend to some bounded, translation invariant operator on $\mathcal{F}^p(\mathbb{C}^d)$. Then by \cref{adjintertw}, $S = BAB^*$ for some bounded operator $A$ on $M^p(\mathbb{R}^d)$ such that $A T_y = T_y A$ for all $y\in\mathbb{R}^d$. By \cref{convchar}, $A$ can be densely defined as 
\begin{align*}
    A f = u*f
\end{align*}
for all $f\in M^1$ for some $u\in M_c(M^p(\mathbb{R}^d))$. We then need to show that $S$ takes the form \cref{Sform}. We can write the action of $S$ on some $F\in \mathcal{F}^1_k(\mathbb{C}^d)$ as 
\begin{align*}
    B(u*B^* F)(z) &= B\, 2^{d/4} \int_{\mathbb{R}^d} \int_{\mathbb{C}^d} u(y)F(w)e^{-\pi(x-y-\overline{w})^2 + \pi\overline{w}^2/2}e^{-\pi|w|^2} \, dy\, dw \\
    &= 2^{d/2} e^{\pi z^2/2} \int_{\mathbb{R}^d} e^{-\pi(x-z)^2} \int_{\mathbb{R}^d} \int_{\mathbb{C}^d} u(y)F(w)e^{-\pi(x-y-\overline{w})^2 + \pi\overline{w}^2/2}e^{-\pi|w|^2} \, dy\, dw\, dx \\
    &= 2^{d/2} e^{-\pi z^2/2}  \int_{\mathbb{R}^d} \int_{\mathbb{C}^d} u(y)F(w)e^{-\pi y^2 - \pi\overline{w}^2/2 -2\pi y\overline{w}} e^{-\pi|w|^2} \int_{\mathbb{R}^d}e^{ - 2\pi (x^2 + xz + xy + x\overline{w}})\, dx \, dy\, dw,
\end{align*}
where we have simply rearranged the exponents from definitions. The last integrand can then be simplified as
\begin{align*}
    \int_{\mathbb{R}^d}e^{ - 2\pi (x^2 + xz + xy + x\overline{w}})\, dx &=  e^{\pi( z + y + \overline{w})^2/2}\int_{\mathbb{R}^d}e^{ - 2\pi(x - (z + y + \overline{w})/2)^2}\, dx \\
    &= 2^{-d/2} e^{\pi( z + y + \overline{w})^2/2},
\end{align*}
and so 
\begin{align*}
    B(u*B^* F)(z) &=e^{-z^2/2}  \int_{\mathbb{R}^d} \int_{\mathbb{C}^d} u(y)F(w)e^{-\pi y^2 - \pi\overline{w}^2/2 -2 \pi y\overline{w} + \pi (z + y + \overline{w})^2/2} e^{-\pi|w|^2} \, dy\, dw \\
    &= \int_{\mathbb{R}^d} \int_{\mathbb{C}^d} u(y)F(w)e^{-\pi y^2/2 - \pi y\overline{w} + \pi zy + \pi z\overline{w}}e^{-\pi|w|^2} \, dy\, dw \\
    &=  \int_{\mathbb{C}^d} F(w)e^{\pi z\overline{w}} e^{-\pi|w|^2} \int_{\mathbb{R}^d}u(y)e^{-\pi y^2/2-\pi y(\overline{w}-z)}\, dy\, dw.
\end{align*}
Hence any bounded, translation invariant operator $S$ on $M^p(\mathbb{R}^d)$ takes the form \cref{Sform}. 
To show the converse, we show that any $S$ of the form \cref{Sform} is translation invariant. This is a simple calculation:
\begin{align*}
    \beta_y SF(z) &= e^{\pi y z-\pi y^2/2} \int_{\mathbb{C}^d} F(w)e^{\pi \overline{w}(z-y)}\phi(\overline{w}-z+y)e^{-\pi |w|^2}\, dw \\
    &= e^{\pi y z-\pi y^2/2} \int_{\mathbb{C}^d} F(u-y)e^{\pi (\overline{u}-y)(z-y)}\phi(\overline{u}-z)e^{-\pi (\overline{u}-y)(u-y)}\, du \\
    &= e^{-\pi y^2/2} \int_{\mathbb{C}^d} F(u-y)e^{\pi\overline{u}z}\phi(\overline{u}-z)e^{\pi uy}e^{-\pi |u|^2}\, du \\
    &= S \beta_y F(z).
\end{align*}
\end{proof}

The results of \cite{CaLiShWiYa20} are formulated on the Fourier side, so the form given for $S$ is in terms of a function $m$ which differs from $\phi$ to that appearing in \cref{Sform}. We confirm that the two approaches are equivalent when one considers the operator as a convolution as opposed to Fourier multiplier as in \cite{CaLiShWiYa20}:
\begin{proposition}
    \Cref{focksob} is equivalent to Theorem 1.1 of \cite{CaLiShWiYa20}, which states that such $S$ are of the form \cref{Sform}, with 
    \begin{align}\label{wickform}
        \phi(z) = 2^{d/2}\int_{\mathbb{R}^d} m(y)e^{-2\pi (y+iz/2)^2}\, dy.
    \end{align}
\end{proposition}
\begin{proof}
    Since the operator $A$ is equivalent to the Fourier multiplier corresponding to $\mathscr{F}u$, we show that taking $m=\mathscr{F}u$ in \cref{wickform} gives our form of $S$:
    \begin{align*}
        \int_{\mathbb{R}^d} m(y)e^{-2\pi (x+iz/2)^2} &= \int_{\mathbb{R}^d}\int_{\mathbb{R}^d} u(x)e^{-2\pi i xy}\, dx \, e^{-2\pi (x+iz/2)^2}\, dy.
    \end{align*}
    We make the change of variables $y=t-\frac{ix}{2}$, to find
    \begin{align*}
        \int_{\mathbb{R}^d}\int_{\mathbb{R}^d} u(x)e^{-2\pi i xy}\, dx \, e^{-2\pi (y+iz/2)^2}\, dy &= \int_{\mathbb{R}^d}\int_{\mathbb{R}^d} u(x)e^{-\pi x^2}e^{-2\pi i xt} e^{-2\pi (-ix/2+t+iz/2)^2}\, \, dx\,dt \\
        &= \int_{\mathbb{R}^d}\int_{\mathbb{R}^d} u(x)e^{-\pi x^2}e^{-2\pi i xt} e^{-2\pi(-x^2/4-ixt - xz/2 + (t+iz/2)^2)}\, \, dx\,dt \\
        &= \int_{\mathbb{R}^d} u(x)e^{-\pi x^2-\pi xz} \int_{\mathbb{R}^d} e^{-2\pi(t+iz/2)^2}\, dt\, dx \\
        &= 2^{-d/2}\int_{\mathbb{R}^d} u(x)e^{-\pi x^2-\pi xz} \, dx. 
    \end{align*}
    and the result follows.
    
\end{proof}
We now aim to extend this characterisation of translation invariant operators to polyanalytic Sobolev-Fock spaces. Key to this step is the observation that
\begin{align}\label{altform}
    \begin{split}
        SF(z) &= \int_{\mathbb{C}^d} F(w)e^{z\overline{w}} e^{-\pi|w|^2} \int_{\mathbb{R}^d}u(y)e^{-\pi y^2/2-\pi y(\overline{w}-z)}\, dy\, dw \\
    &= \int_{\mathbb{R}^d}u(y)e^{-\pi y(y/2-z)} \int_{\mathbb{C}^d} F(w)e^{(z-y)\overline{w}} e^{-\pi|w|^2}\, dw \, dy \\
    &= \int_{\mathbb{R}^d} F(z-y) u(y)e^{-\pi y(y/2-z)}\, dy \\
    &= \int_{\mathbb{R}^d} \beta_y F(z) u(y)\, dy.
    \end{split}
\end{align}
This is perhaps unsurprising, since using the intertwining of $\beta_z$ and $B$ along with the unitarity of the Bargmann transform;
\begin{align*}
\begin{split}
    B(u*B^* F) &= B\Big(\langle u(\cdot),  \pi(\cdot,0)B^*F(x)\rangle_{L^2} \Big) \\
    &= B \Big( \int_{\mathbb{R}^d} u(y)\pi(y,0)B^* F(x)\, dy \Big) \\
    &= \int_{\mathbb{R}^d} u(y)B\pi(y,0)B^* F(x)\, dy \\
    &= \int_{\mathbb{R}^d} \beta_y F(z) u(y)\, dy,
\end{split}
\end{align*}
however, it is interesting to note that the form of \cref{Sform} can be retrieved simply by considering the intertwining of the translation operator $\beta_y$ and Bargmann transform $B$. This suggests that using \cref{convchar}, we can proceed to characterise the translation invariant operators on Polyanalytic Sobolev-Fock spaces by understanding the intertwining of shift operator $\beta_z$ with the polyanalytic Bargmann transform. This behaviour is described by the following proposition:

\begin{proposition}\label{intertwk}
    Given $f\in M^p(\mathbb{R}^d)$ for $1\leq p \leq \infty$, and $k\in\mathbb{N}$;
    \begin{align*}
        B^k \pi(z) f = \beta_{z} B^k f.
    \end{align*}
\end{proposition}
This is simply the intertwining property of $\beta$ for the case $k=0$, it remains to prove the generalisation for all $k$. To do so we will use the following lemma, an application of the differentiation property \cref{diffformula}:
\begin{lemma}\label{semireprodkern}
    Given $f\in M^p(\mathbb{R}^d)$ for $1\leq p \leq \infty$, and $k\in\mathbb{N}$;
    \begin{align*}
        B^k f(z) &= c_k \langle F(w), e^{\pi \overline{z}w}(w-z)^k\rangle_{\mathcal{F}^2} 
    \end{align*}
    where $F(z)=Bf(z)$ and $c_k = \frac{1}{\sqrt{k! \pi^{|k|}}}$.
\end{lemma}
\begin{proof}
    Expanding the binomial gives
    \begin{align*}
        \langle F(w), e^{\pi \overline{z}w}(w-z)^k\rangle_{\mathcal{F}^2} &= c_k \sum_{l\leq k} \binom{k}{l} (-\overline{z})^l \langle F(w), e^{\pi \overline{z}w}w^{k-l}\rangle_{\mathcal{F}^2}.
    \end{align*}
    Using the differentiation formula \cref{diffformula} we find
    \begin{align*}
        c_k \sum_{l\leq k} \binom{k}{l} (-\overline{z})^l \langle F(w), e^{\pi \overline{z}w}w^{k-l}\rangle_{\mathcal{F}^2} &= c_k \sum_{l\leq k} \binom{k}{l} (-\pi\overline{z})^l F^{k-l}(z) \\
        &= c_k e^{\pi |z|^2} \sum_{l\leq k} \binom{k}{l}e^{-\pi |z|^2} (-\pi\overline{z})^l F^{k-l}(z) \\
        &= c_k e^{\pi |z|^2} \frac{d^k}{dz^k}\Big( e^{-\pi |z|^2} F(z) \Big) \\
        &= B^k f(z)
    \end{align*}
\end{proof}
We can then prove the proposition:
\begin{proof}[Proof of \cref{intertwk}]
    Let $F(z)=Bf(z)$, and denote $\xi=y+i\eta$ and $c_k = \frac{1}{\sqrt{k! \pi^{|k|}}}$. Then from \cref{semireprodkern};
    \begin{align*}
        B^k \pi(y,\eta)f(z) &= c_k \langle \beta_{\xi}F(w), e^{\pi \overline{z}w}(w-z)^k\rangle_{\mathcal{F}^2} \\
        \\
        &= c_k e^{2\pi i y\eta}\langle F(w), \beta_{-\xi}e^{\pi \overline{z}w}(w-z)^k\rangle_{\mathcal{F}^2} \\
        &= c_k e^{2\pi i y\eta}\langle F(w), e^{i\pi y\eta - \pi |\xi|^2/2- \pi \overline{\xi}w} e^{\pi \overline{z}(w+\xi)}(w+\xi-z)^k\rangle_{\mathcal{F}^2} \\
        &= c_k e^{\pi i y\eta - \pi |\xi|^2/2 + \pi \overline{\xi}z}\langle F(w), e^{- \pi \overline{\xi}w} e^{\pi \overline{z}w}(w+\xi-z)^k\rangle_{\mathcal{F}^2} \\
        &= \beta_{\xi} B^k f(z)
    \end{align*}
\end{proof}
We remark that as in \cref{adjintertw}, a simple calculation using the unitarity of $B^k$ gives $(B^k)^*\beta_z F = \pi(z)(B^k)^* F$. From here we can proceed to characterise translation invariant operators on Polyanalytic Fock spaces:
\begin{theorem}\label{polyfocksob}
    Let 
    \begin{align*}
        \phi(z,w) := \int_{\mathbb{R}^d}(y-\overline{z})^k u(y)e^{-\pi y^2/2-\pi yw}\, dy,
    \end{align*}
    for some $u\in M_c(M^p(\mathbb{R}^d)$, $1\leq p < \infty$ and $k\in\mathbb{N}$. Then $S:\mathcal{F}_k^1\to\mathcal{F}_k^{\infty}$, defined by
    \begin{align*}\label{Spolyform}
        S F (z) = \int_{\mathbb{C}^d} \Tilde{F}(w)e^{\pi \overline{w}z} \phi(z-w,z-\overline{w}) e^{-\pi |w|^2}\, dw,
    \end{align*}
    for $F\in \mathcal{F}^1(\mathbb{C}^d)$, where $\Tilde{F}=B(B^k)^* F$, extends to a bounded, translation invariant operator on $\mathcal{F}^p_k(\mathbb{C}^d)$.
    Alternatively we can write $S$ as
    \begin{align}
        S F (z) = \int_{\mathbb{C}^d} F(w)e^{\pi \overline{w}z} \psi(z-w,z-\overline{w}) e^{-\pi |w|^2}\, dw,
    \end{align}
    where
    \begin{align*}
        \psi(z,w) = \int_{\mathbb{R}^d} L_{k-1}(\pi |z-y|^2) u(y)e^{-\pi y^2/2-\pi yw}\, dy.
    \end{align*}
    Conversely, if $S$ is a bounded and translation invariant operator on $\mathcal{F}^p_k(\mathbb{C}^d)$, then $S$ is of the form \cref{Spolyform}, with corresponding $u\in M_c(M^p(\mathbb{R}^d)$.
\end{theorem}
\begin{proof}
    As in \cref{focksob}, let Let $S:\mathcal{F}^1_k(\mathbb{C}^d)\to \mathcal{F}^{\infty}_k(\mathbb{C}^d)$ extend to some bounded, translation invariant operator on $\mathcal{F}^p(\mathbb{C}^d)$, and consider $A=B^*SB$. Then there exists some $u\in M_c(M^p(\mathbb{R}^d)$ such that
    \begin{align*}
        A f = u*f
    \end{align*}
    for $f\in M^1(\mathbb{R}^d)$, which extends to a bounded operator on $M^p(\mathbb{R}^d)$ by \cref{convchar}. We then use the intertwining property for $B^k$;
    \begin{align*}
        B^k(u*(B^k)^* F) &= B^k\Big(\langle u(\cdot),  \pi(\cdot,0)(B^k)^*F(x)\rangle \Big) \\
        &= B^k \Big( \int_{\mathbb{R}^d} u(y)\pi(y,0)(B^k)^* F(x)\, dy \Big) \\
        &= \int_{\mathbb{R}^d} u(y)B^k\pi(y,0)(B^k)^* F(x)\, dy \\
        &= \int_{\mathbb{R}^d} \beta_y F(z) u(y)\, dy \\
        &= \int_{\mathbb{R}^d} F(z-y) e^{-\pi y(y/2 - z)}u(y)\, dy.
    \end{align*}
    we can then substitute the formula of \cref{semireprodkern}:
    \begin{align*}
        \int_{\mathbb{R}^d} F(z-y) e^{-\pi y(y/2 - z)}u(y)\, dy &= \int_{\mathbb{C}^d} \int_{\mathbb{R}^d} \Tilde{F}(w) e^{\pi (z-y) \overline{w}} (\overline{w}+y-\overline{z})^k e^{-\pi y(y/2 - z)}u(y)\, dy \, e^{-\pi|w|^2}\, dw \\
        &= \int_{\mathbb{C}^d} \Tilde{F}(w) e^{\pi z \overline{w}} \int_{\mathbb{R}^d} (\overline{w}+y-\overline{z})^k e^{-\pi y(y/2 - z + \overline{w})}u(y)\, dy \, e^{-\pi|w|^2}\, dw \\
        &= \int_{\mathbb{C}^d} \Tilde{F}(w) e^{\pi z \overline{w}} \phi(z,w) \, e^{-\pi|w|^2}\, dw
    \end{align*}
    for the first part of the statement. For the second, we simply use the reproducing formula for $\mathcal{F}^p_k(\mathbb{C}^d)$:
    \begin{align*}
         \int_{\mathbb{R}^d} F(z-y) e^{-\pi y(y/2 - z)}u(y)\, dy &= \int_{\mathbb{C}^d} \int_{\mathbb{R}^d} F(w) e^{\pi (z-y) \overline{w}} L_{k-1}(\pi |w+y-z|^2) e^{-\pi y(y/2 - z)}u(y)\, dy \, e^{-\pi|w|^2}\, dw \\
         &= \int_{\mathbb{C}^d} F(w)e^{\pi \overline{w}z} \psi(z-w,z-\overline{w})e^{-\pi |w|^2}\, dw.
    \end{align*}
    Conversely, if $S$ is of the form \cref{Spolyform}, then
    \begin{align*}
        \beta_y SF(z) &= e^{\pi y z-\pi y^2/2} \int_{\mathbb{C}^d} F(w)e^{\pi \overline{w}(z-y)}\phi(z-w-y,z-\overline{w}-y)e^{-\pi |w|^2}\, dw \\
        &= e^{\pi y z-\pi y^2/2} \int_{\mathbb{C}^d} F(u-y)e^{\pi (\overline{u}-y)(z-y)}\phi(z-u,z-\overline{u})e^{-\pi (\overline{u}-y)(u-y)}\, du \\
        &= e^{-\pi y^2/2} \int_{\mathbb{C}^d} F(u-y)e^{\pi\overline{u}z}\phi(z-u,z-\overline{u})e^{\pi uy}e^{-\pi |u|^2}\, du \\
        &= S \beta_y F(z).
\end{align*}
\end{proof}

\section{Frame Representations of Translation Invariant Operators}
Identified with the Gabor spaces $V_{\varphi_k}(M^p(\mathbb{R}^d))$, the Polyanalytic Sobolev-Fock spaces lend themselves naturally to frame representations. In particular, given the continuous frame given by $\{\pi(z)\varphi_k\}_{z\in\mathbb{R}^{2d}}$, we can consider the action of some operator $T:M^1\to M^{\infty}$ by the function
\begin{align*}
    G_T(w,z) := \langle T\pi (w)\varphi_k, \pi(z) \varphi_k\rangle_{\mathscr{S}^\prime,\mathscr{S}}
\end{align*}
on double phase space. On the diagonal $w=z$ this is the Berezin transform, but by considering off-diagonal decay we can utilise other results from time-frequency analysis. In particular, Gröchenig proved \cite{Gr06} that if there exists some $H\in L^1(\mathbb{R}^{2d})$ such that 
\begin{align*}
    |G(w,z)| \leq H(w-z)
\end{align*}
for all $w,z\in \mathbb{R}^{2d}$, then $T$ is a bounded operator on $M^p(\mathbb{R}^d)$ for all $1\leq p \leq \infty$ \cite{Gr06}. In \cite{Zh15}, the boundedness of the Berezin transform $\langle Sk_z, k_z\rangle_{\mathcal{F}^2}$ was noted as a necessary condition for the boundedness of $S$ on $\mathcal{F}^2(\mathbb{C}^d)$. By considering the polarised version, we can hence give a sufficient condition for $S$ to be bounded on \textit{all} spaces $\mathcal{F}^p(\mathbb{C}^d)$ as well. 
\begin{theorem}\label{modspacecond}
    Given some operator $S: \mathcal{F}^1(\mathbb{C}^d)\to \mathcal{F}^{\infty}(\mathbb{C}^d)$ with the form
    \begin{align*}
        S F=\int F(w)e^{\pi\overline{z}w}\phi(\overline{w}-z)e^{-\pi |w|^2}\, dw
    \end{align*}
    for some $\phi \in \mathscr{S}'(\mathbb{R}^d)$, then $S$ is bounded on $\mathcal{F}^p(\mathbb{C}^d)$ for all $1\leq p \leq \infty$ if there exists some $H\in L^1(\mathbb{C}^{d})$ such that
    \begin{align*}
        \big| \phi(\overline{z}-w) \big| \leq \big|e^{\pi(|z|^2+|w|^2)/2 - \pi w\overline{z}}\big|\cdot H(z-w).
    \end{align*}
\end{theorem}
\begin{proof}
Since we have from our intertwining relations that 
\begin{align*}
    \Tilde{k}_z(w) = e^{-i\pi x\omega}B(\pi(\overline{z})\varphi_0)(w),
\end{align*}
where $\Tilde{k}_z$ denotes the normalised kernel $e^{-\pi|z|^2/2}k_z$,
it follows that
\begin{align}\label{gabmatr}
\begin{split}
    \langle S\Tilde{k}_z, \Tilde{k}_w\rangle_{\mathcal{F}^{\infty},\mathcal{F}^1} = e^{-i\pi( x\omega - y\eta)} \langle T\pi (\overline{z})\varphi_0, \pi(\overline{w}) \varphi_0\rangle_{M^{\infty},M^1}
\end{split}
\end{align}
where $z=(x,\omega)$, $w=(y,\eta)$ and $T=B^*SB$. We can calculate directly the form of $\langle S\Tilde{k}_z, \Tilde{k}_w\rangle_{\mathcal{F}^{\infty},\mathcal{F}^1}$ for $S F=\int F(w)e^{\pi\overline{z}w}\phi(\overline{w}-z)e^{-\pi |w|^2}\, dw$ as above. To do so, we use the form of $S$ from \cref{altform} to see
\begin{align*}
    S \Tilde{k}_z(\xi) &= e^{-\pi|z|^2/2}\int_{\mathbb{R}^d} e^{\pi \overline{z}(\xi-y)}u(y)e^{-\pi y^2/2+y\xi }\, dy \\
    &= e^{-\pi|z|^2/2}\int_{\mathbb{R}^d} u(y)e^{\pi \xi (\overline{z}+y) -\pi y^2/2 - y\overline{z} }\, dy,
\end{align*}
and so
\begin{align}\label{gabcond}
\begin{split}
     e^{\pi(|z|^2+|w|^2)/2}\langle S\Tilde{k}_z, \Tilde{k}_w\rangle_{\mathcal{F}^{\infty},\mathcal{F}^1} &= \big\langle \int_{\mathbb{R}^d} u(y)e^{\pi \xi (\overline{z}+y) -\pi y^2/2 - y\overline{z}}\, dy, e^{\pi \xi\overline{w}}\big\rangle_{\mathcal{F}^{\infty},\mathcal{F}^{1}} \\
    &= \int_{\mathbb{R}^d} u(y) \langle e^{\pi \xi (\overline{z}+y) -\pi y^2/2 - y\overline{z}}, e^{\pi \xi\overline{w}}\rangle_{\mathcal{F}^{\infty},\mathcal{F}^1}\, dy \\
    &= \int_{\mathbb{R}^d} u(y)e^{\pi w (\overline{z}+y) -\pi y^2/2- y\overline{z}}\, dy \\
    &= e^{\pi w\overline{z}} \phi(\overline{z}-w).
\end{split}
\end{align}
Comparing \cref{gabmatr} with \cref{gabcond}, we have that
\begin{align*}
    |\langle T\pi (\overline{z})\varphi_0, \pi(\overline{w}) \varphi_0\rangle_{M^{\infty},M^1}| = \big| e^{\pi w\overline{z} -\pi(|z|^2+|w|^2)/2} \phi(\overline{z}-w) \big|,
\end{align*}
and so the statement follows from the result of \cite{Gr06} on $M^p(\mathbb{R}^d)$ spaces.

\end{proof}
\begin{corollary}
    If $S$ is as in \cref{modspacecond} and is bounded on $M^p$ for $1\leq p \leq \infty$, then $\phi(\overline{z}-z)$ is bounded.
\end{corollary}
\begin{proof}
    This follows by taking $z=w$ in \cref{modspacecond} to find
    \begin{align*}
        |\phi(\overline{z}-z)| \leq e^0 H(0).
    \end{align*}
\end{proof}
Considering now operators on $\mathcal{F}^p_k$, we can follow the same procedure to find:
\begin{theorem}\label{modspacecondpoly}
    Given some operator $S: \mathcal{F}_k^1(\mathbb{C}^d)\to \mathcal{F}_k^{\infty}(\mathbb{C}^d)$ with the form
    \begin{align*}
        S F(z)=\int F(w)e^{\pi\overline{z}w}\psi(z-w,z-\overline{w})e^{-\pi |w|^2}\, dw
    \end{align*}
    for some $\psi \in \mathscr{S}'(\mathbb{R}^{2d})$, then $S$ is bounded on $\mathcal{F}_k^p(\mathbb{C}^d)$ for all $1\leq p \leq \infty$ if there exists some $H\in L^1(\mathbb{C}^{d})$ such that
    \begin{align*}
        \big| \psi(z-w,\overline{z}-w) \big| \leq \big|e^{\pi(|z|^2+|w|^2)/2 - \pi w\overline{z}}\big|\cdot H(z-w).
    \end{align*}
\end{theorem}
\begin{proof}
    Again we note that \cref{modspacecond} relies on the intertwining property of $\beta$ and $B$, along with the unitarity of $B$. These are two features which extend to the polyanalytic case, and so we proceed in the same manner to find that
    \begin{align*}
        \langle S \Tilde{k}_z, \Tilde{k}_w\rangle_{\mathcal{F}_k^{\infty},\mathcal{F}_k^n} = e^{-i\pi(x\omega-y\eta)}\langle T\pi(\overline{z})\varphi_k,\pi(\overline{w})\varphi_k\rangle_{M^{\infty},M^1}.
    \end{align*}
    We find 
    \begin{align*}
        S \Tilde{k}_z(\xi) = \int_{\mathbb{R}^d}\Tilde{k}_z(\xi-y)u(y)e^{-\pi y^2/2 + \pi y\xi}\, dy
    \end{align*}
    and so
    \begin{align*}
        \langle S \Tilde{k}_z, \Tilde{k}_w\rangle_{\mathcal{F}_k^{\infty},\mathcal{F}_k^n} &= \Big\langle \int_{\mathbb{R}^d}\Tilde{k}_z(\xi-y)u(y)e^{-\pi y^2/2 + \pi y\xi}\, dy, \Tilde{k}_w \rangle_{\mathcal{F}_k^{\infty},\mathcal{F}_k^n} \\
        &= e^{-\pi |w|^2/2}\int_{\mathbb{R}^d}\Tilde{k}_z(w-y)u(y)e^{-\pi y^2/2 + \pi yw}\, dy \\
        &= e^{-\pi(|z|^2 + |w|^2)/2}\int_{\mathbb{R}^d}\pi L_k(\pi |z-w-y|^2)e^{\pi \overline{z} (w-y)} u(y)e^{-\pi y^2/2+\pi yw}\, dy \\
        &= e^{-\pi(|z|^2 + |w|^2)/2 + \pi w\overline{z}} \psi(z-w,z-\overline{w}),
    \end{align*}
    and so the result follows.

\end{proof}
These sufficient conditions for operators bounded on $\mathcal{F}_k^p$ spaces give an insight into a broader class of operators considered in \cite{CaLiShWiYa20}. The authors asked which functions $K_T(z,\overline{w})$ are the integral kernels of bounded maps on $\mathcal{F}^2(\mathbb{C}^d)$, where
\begin{align*}
    K_T(z,\overline{w}) = TK(\cdot,\overline{w})(z)
\end{align*}
for some bounded operator $T$ on $\mathcal{F}^2(\mathbb{C}^d)$. It is clear from definitions that
\begin{align*}
    \langle T \Tilde{k}_z,\Tilde{k}_w\rangle = e^{-\pi(|z|^2+|w|^2)/2} K_T(z,\overline{w}).
\end{align*}
Hence from the same line of reasoning as above, we have the following sufficient condition:
\begin{proposition}
    Let $K_T(z,\overline{w})\in \mathscr{S}'(\mathbb{C}^{2d})$ such that there exists some $H\in L^1(\mathbb{C}^d)$ satisfying
    \begin{align*}
        \big| K_T(z,\overline{w}) \big| \leq \big|e^{\pi(|z|^2+|w|^2)/2}\big|\cdot H(z-w),
    \end{align*}
    then $K_T$ is a bounded integral operator on all $\mathcal{F}^p(\mathbb{C}^d)$ spaces. In particular it is a bounded operator on the Fock space.
\end{proposition}

\section{Examples}
We consider here some examples of operators bounded on different Sobolev-Fock spaces.
\begin{example}\normalfont
The motivating example of \cite{Zh15} was the unitary equivalent of the Hilbert transform to the operator on the Fock space;
\begin{align*}
    S_{A}F := BHB^* F = \int_{\mathbb{C}^d} F(w)e^{\pi z \overline{w}}A\Big(\frac{z-\overline{w}}{\sqrt{2}}\Big)e^{-\pi |w|^2}\, dw,
\end{align*}
where $A(z)$ is the antiderivative of $e^{u^2}$ at $z$ satisfying $A(0)=0$. It was noted in \cite{CaLiShWiYa20} that $S_A$ is \textit{not} bounded on the spaces defined as the analytic $F$ such that
\begin{align*}
    \int_{\mathbb{C}^d} |F(z)|^p e^{-\pi |z|^2}\, dz < \infty
\end{align*}
for $1\leq p<2$. The authors remark that one may consider the Sobolev-Fock space treated in this work, but that $S_H$ would not be well-defined for $p>2$, however associating the Bargmann transform with the STFT shows this is not an issue. With our definition of the Sobolev-Fock spaces, we can use the  specific instance of Theorem 1 of \cite{BeGrOkRo07}, which states that the Hilbert transform is bounded on $M^p(\mathbb{R}^d)$ for $1<p<\infty$. From this it follows that $S_H$ is bounded on all $\mathcal{F}^p(\mathbb{C}^d)$ for $1<p<\infty$. By \cref{modspacecond}, the failure of $S_H$ to be bounded on the endpoints $\mathcal{F}^1$, $\mathcal{F}^{\infty}$ informs us that $A(z-\overline{w})e^{-\pi(|z| + |w|)/2-\pi\overline{z}w}$ cannot be dominated by $H(z-w)$ for some $H\in L^1(\mathbb{C}^d)$, otherwise by \cref{modspacecond} $S_H$ would be bounded for $p=1,\infty$.
\end{example}

\begin{example}\normalfont
We now consider an example of a convolution operator which is bounded on all modulation spaces, whilst being defined only on $L^2(\mathbb{R}^d)$ of the Lesbesgue spaces. In \cite{BeGrOkRo07} the authors proved that $e^{-i t\pi y^2}$ defines a bounded convolution operator on $M^p(\mathbb{R}^d)$ for all $1\leq p \leq \infty$ and real constant $t$, yet fails to be bounded on any $L^p(\mathbb{R}^d)$ for $p\neq 2$. Constructing the corresponding operator for the $t=1$ case on Sobolev-Fock spaces gives
\begin{align*}
    SF(z) = \int_{\mathbb{C}^d}F(w)e^{\pi z\overline{w}}\phi(z-\overline{w})e^{-\pi |w|^2}\, dw,
\end{align*}
where
\begin{align*}
    \phi(z) = \int_{\mathbb{R}^d} e^{\pi y((i+1/2)y - z)} \, dy.
\end{align*}
Hence the operator 
\begin{align*}
    SF(z) = \int_{\mathbb{R}^d} \int_{\mathbb{C}^d} F(w)e^{\pi (z\overline{w} + c_1 (z-\overline{w})^2 - |w|^2)}\, dw
\end{align*}
is bounded on all $\mathcal{F}^p(\mathbb{C}^d)$ spaces for $1\leq p \leq \infty$. The complex constant $c_1$ can be manipulated by choice of $t$.
\end{example}

\section*{Acknowledgements}
The author would like to thank Franz Luef and Monika Dörfler for their helpful comments.

\bibliography{refs}
\bibliographystyle{abbrv}

\end{document}